\documentclass[12pt,a4paper,reqno]{amsart}

\usepackage[margin=2cm]{geometry}
\usepackage{amssymb}
\usepackage{amsmath,amsthm}
\usepackage{amsfonts}
\usepackage{mathrsfs}

\usepackage{color}
\usepackage{hyperref}
\hypersetup{
	colorlinks = true,%
	citecolor = blue,
	filecolor=red,%
	linkcolor = [rgb]{0.65,0.0,0.0},%
	anchorcolor = red,
	pagecolor = red,
	urlcolor= [rgb]{0.65,0.0,0.0}
	linktocpage=true,
	pdfpagelabels=true,
	bookmarksnumbered=true,
}

\newcommand{\R}{{\mathbb R}}

\newcommand{\Pl}{(\mathcal{P}_{\lambda})}

\newcommand{\ds}{\displaystyle}

\newtheorem{theor}{Theorem}[section]

\newtheorem{lem}{Lemma}[section]

\newtheorem{question}{Question}[section]
\newtheorem{rem}{Remark}[section]

\title[Kirchhoff-type problems involving the critical Sobolev exponent]{On an open question of Ricceri concerning a Kirchhoff-type problem}
\author{Francesca Faraci }

\email{ffaraci@dmi.unict.it}

\address{Department of Mathematics and Computer Science, University of Catania,
	Catania, Italy}

\author{Csaba Farkas}

\email{farkas.csaba2008@gmail.com \& farkascs@ms.sapientia.ro}

\address{Department of Mathematics and Computer Science, Sapientia University,
	Tg. Mures, Romania \& Institute of Applied Mathematics, Obuda University,
1034 Budapest, Hungary}

\begin{document}
\maketitle	
	
\begin{abstract} In the present note we prove a multiplicity result for a Kirchhoff type  problem involving a critical term,   giving a partial positive answer to a problem raised by Ricceri.
	\end{abstract}
\section{Introduction}	
Nonlocal boundary value problems of the type
$$
\left\{
  \begin{array}{ll}
    - \left( a+b\ds\int_\Omega |\nabla u|^2 dx\right)\Delta u=
f(x,u), & \hbox{ in } \Omega \\ \\
    u=0, & \hbox{on } \partial \Omega
  \end{array}
\right.
$$
are related to the stationary version of the Kirchhoff equation
$$\frac{\partial^2 u}{\partial t^2}- \left( a+b\ds\int_\Omega |\nabla u|^2 dx\right)\Delta u=f_1(t,x,u),$$ first proposed by Kirchhoff to describe the transversal oscillations of a stretched string. Here $\Omega$ is a bounded domain of $\R^N$, $u$ denotes the displacement, $f_1$ is the
external force, $b$ is the initial tension and $a$ is related to the intrinsic properties
of the string.

 Note that, this type of nonlocal equations appears in other fields like biological systems, where $u$ describes a process depending on the average of itself, like population density (see for instance \cite{CL}).

The first attempt to find solutions for subcritical nonlinearities, by means of variational methods, is due to Ma and Rivera \cite{MR} and  Alves, Corr\^{e}a and Ma \cite{ACM} who combined minimization arguments with truncation techniques and a priori estimates. Using Yang index and critical group arguments or the theory of invariant sets of descent flows, Perera and Zhang (see \cite{PZ,ZP}) proved existence results for the above problem. Multiplicity theorems can be found for instance in \cite{CKW,MZ,R0}.

 The existence or multiplicity of solutions of the Kirchhoff type problem  with critical exponents in a bounded domain (or even in the whole space) has been studied by using different techniques as  variational methods, genus theory,  the Nehari manifold, the Ljusternik--Schnirelmann category theory (see for instance \cite{CF,Fan,F,FS}).
  It is worth mentioning that  Mountain Pass arguments combined with  the Lions' Concentration Compactness principle \cite{L} are still the most popular tools to deal with such problems in the presence of a critical term. Applications to the lower dimensional case ($N<4$) can be found in  \cite{ACF,LLG,N}, while for  higher dimensions  ($N\geq4$) we refer to \cite{H1,H2,N0,YM}. Notice that in order to employ the Concentration Compactness principle, $a$ and $b$ need to satisfy suitable constraints.

In order to state our main result we introduce the following notations: we endow the Sobolev space $H^1_0(\Omega)$ with the
classical norm $\|u\|=\left( \int_{\Omega }|\nabla u|^2 \ dx\right)^{\frac{1}{2}}$ and  denote by $\|u\|_{q}$ the Lebesgue norm in $L^{q}(\Omega)$ for $1\leq q \leq 2^\star$,  i.e. $\|u\|_{q}=\left(\int_{\Omega} |u|^{q} \ dx\right)^{\frac{1}{q}}$.
Let $S_N$ be the embedding constant of $H^1_0(\Omega)\hookrightarrow L^{2^\star}(\Omega)$, i.e.
\[\|u\|^2_{2^\star}\leq S_N^{-1} \|u\|^2 \qquad \mbox{for every } \ u\in H^1_0(\Omega). \]
Let us recall that (see Talenti \cite{Talenti} and Hebey \cite{H1} for an explicit espression)
\begin{equation}\label{2*}
S_N=\frac{N(N-2)}{4}\omega_N^{\frac{2}{N}},
\end{equation}
where $\omega_N$ is the volume of the unit ball in $\R^N$.
For $N\geq4$ denote by $C_1(N)$ and $C_2(N)$ the constants
\[
C_1(N)=
\begin{cases}\ds
\frac{4(N-4)^{\frac{N-4}{2}}}{N^{\frac{N-2}{2}}S_{N}^{\frac{N}{2}}} & N>4\\ \\
\ds \frac{1}{S_{4}^{2}}, & N=4,
\end{cases}
\qquad \mbox{ and   } \qquad
C_2(N)=\begin{cases}
\ds\frac{2(N-4)^{\frac{N-4}{2}}}{(N-2)^{\frac{N-2}{2}}S_{N}^{\frac{N}{2}}} & N>4\\ \\
\ds \frac{1}{S_{4}^{2}}, & N=4.
\end{cases}.
\]
Notice that $C_1(N)\leq  C_2(N)$.

Our result reads as follows:

\begin{theor}\label{our theorem} Let $a, b$  be positive numbers, $N\ge4$. \\
\noindent	(A)  If $ a^{\frac{N-4}{2}} b\geq C_1(N)$, then, for each $\lambda>0$ large enough and for each convex set $C\subseteq L^2(\Omega)$ whose closure in $L^2(\Omega)$ contains $H^1_0(\Omega)$, there exists $v^*\in C$ such that the functional
	\[u\to \frac{a}{2} \int_{\Omega}|\nabla u|^2 dx +\frac{b}{4} \left( \int_{\Omega}|\nabla u|^2 dx\right)^2-\frac{1}{2^\star}\int_{\Omega}|u|^{2^\star} dx-\frac{\lambda}{2}\int_{\Omega}|u(x)-v^*(x)|^2 dx\] has two global minima.
	
	\noindent (B) If $a^{\frac{N-4}{2}} b> C_2(N)$,
	then,
	for each $\lambda>0$ large enough and for each convex set $C\subseteq L^2(\Omega)$ whose closure in $L^2(\Omega)$ contains $H^1_0(\Omega)$, there exists $v^*\in C$ such that the problem
	
	$$
	\left\{
	\begin{array}{ll}
	- \left( a+b\ds\int_\Omega |\nabla u|^2 dx\right)\Delta u=
	|u|^{2^\star-2}u+\lambda (u-v^*(x)), & \hbox{ in } \Omega \\ \\
	u=0, & \hbox{on } \partial \Omega
	\end{array}
	\right.\eqno{(\mathcal{P}_{\lambda})}
	$$
	has at least three weak solutions, two of which are global minima in $H^1_0(\Omega)$ of the energy functional defined in (A).
\end{theor}

The paper is motivated by a recent work of Ricceri  where the author studied problem $(\mathcal P_\lambda)$ in the subcritical case, i.e. when $|u|^{2^\star-2}u$ is replaced by $|u|^{p-2}u$ with $p<2^\star$. In \cite[Proposition 1]{R},  the existence of two global minima  for the energy functional (and three solutions for the associated  Kirchhoff problem) is obtained for every $a\geq 0$ and $b>0$. In the same paper, the following challenging question was raised (see \cite[Problem 1]{R}):
\begin{question}
	Does the  conclusion of Proposition 1  hold if $N>4$ and $p=2^\star$?
\end{question}

Notice that, for $N>4$ and $p=2^\star$ the energy functional associated to $\Pl$ is bounded from below while if $N=4(=2^\star)$ this is not true any more for arbitrary $b$. Moreover, when $p=2^\star$ the embedding of $H^1_0(\Omega)$  into $L^{p}(\Omega)$ fails to be compact and one can not apply directly the abstract tool which leads to \cite[Theorem 1 \& Proposition 1]{R}.

	The main result of the present note gives a partial positive answer to the above question and prove that Proposition 1 of \cite{R} holds for $p=2^\star$ and  $N\geq 4$   provided that $a$ and $b$ satisfies a suitable crucial inequality.
Namely, we prove that the interaction between the Kirchhoff type operator and the critical nonlinearity ensures the sequentially weakly lower semicontinuity of the energy functional, a key property which allows to apply the minimax theory developed in  \cite[Theorem 3.2]{R1} (see also Theorem \ref{minimax theorem} below).

\section{Proofs}

The proof of Theorem \ref{our theorem} relies on the following key lemma  (see  \cite{FFK} for a deeper study on this topic).

\begin{lem}\label{semicontinuity}
	Let $N\geq 4$ and $a, b$  be positive numbers such that $ a^{\frac{N-4}{2}} b\geq C_1(N)$. Denote by $\mathcal F:H^1_0(\Omega)\to\R$ the functional
	\[\mathcal F(u)=\frac{a}{2}\|u\|^2+\frac{b}{4} \|u\|^4-\frac{1}{2^\star}\|u\|^{2^\star}_{2^\star} \qquad \mbox{for every }\  u \in H^1_0(\Omega).\]
	Then, $\mathcal F$ is sequentially weakly lower semicontinuous in $H^1_0(\Omega)$.
	\end{lem}
	
	\begin{proof}
		Fix $u \in H^1_0(\Omega)$ and let $\{u_n\} \subset H^1_0(\Omega)$ such that $u_n\rightharpoonup u$ in $H^1_0(\Omega)$. Thus,
		\begin{align*}
		\mathcal{F}(u_n)-\mathcal{F}(u) =&\frac{a}{2}(\|u_n\|^2-\|u\|^2)+\frac{b}{4}(\|u_n\|^4-\|u\|^4)\\ &-\frac{1}{2^\star}\left(\|u_n\|_{2^\star}^{2^\star}-\|u\|_{2^\star}^{2^\star}\right).
		\end{align*}
		It is clear that \begin{align*}\|u_n\|^2-\|u\|^2&=\|u_n-u\|^2+2\int_{\Omega}\nabla(u_n-u)\nabla u \\ &= \|u_n-u\|^2+o(1),
		\end{align*}
		and
		\begin{align*}
		\|u_n\|^4-\|u\|^4&=\left(\|u_n-u\|^2+o(1)\right)\left(\|u_n-u\|^2+2\int_{\Omega}\nabla u_n\nabla u\right)\\&=\left(\|u_n-u\|^2+o(1)\right)\left(\|u_n-u\|^2+2\int_{\Omega}\nabla (u_n-u)\nabla u+2\|u\|^2\right)\\
		&=\left(\|u_n-u\|^2+o(1)\right)\left(\|u_n-u\|^2+2\|u\|^2+o(1)\right).
		\end{align*}
		Moreover, from the Br\'ezis-Lieb lemma, one has
		
	 $$\|u_n\|_{2^\star}^{2^\star}-\|u\|_{2^\star}^{2^\star}=\|u_n-u\|_{2^\star}^{2^\star}+o(1).$$
		
		Putting together the above outcomes,
		\begin{align*}
		\mathcal{F}(u_n)-\mathcal{F}(u)=&\frac{a}{2}\|u_n-u\|^2+\frac{b}{4}\left(\|u_n-u\|^4+2\|u\|^2\|u_n-u\|^2\right)-\frac{1}{2^\star}\|u_n-u\|_{2^\star}^{2^\star}+o(1) \\{\geq}& \frac{a}{2}\|u_n-u\|^2+\frac{b}{4}\left(\|u_n-u\|^4+2\|u\|^2\|u_n-u\|^2\right)-\frac{{S}_N^{-\frac{2^\star}{2}}}{2^\star}\|u_n-u\|^{2^\star}+o(1)
		\\ \geq& \frac{a}{2}\|u_n-u\|^2 +\frac{b}{4}\|u_n-u\|^4-\frac{{S}_N^{-\frac{2^\star}{2}}}{2^\star}\|u_n-u\|^{2^\star}+o(1)\\=& \|u_n-u\|^2 \left(\frac{a}{2}+\frac{b}{4}\|u_n-u\|^2-\frac{{S}_N^{-\frac{2^\star}{2}}}{2^\star}\|u_n-u\|^{2^\star-2}\right)+o(1).
		\end{align*}
		
		Denote by $f:[0,+\infty[\to\R$  the  function $\displaystyle f(x)=\frac{a}{2}+\frac{b}{4}x^2-\frac{{S}_N^{-\frac{2^\star}{2}}}{2^\star}x^{2^\star-2}$. We claim that $f(x)\geq 0$ for all $x\geq 0$.
		
		\bigskip
		
		Indeed, when $N=4$, and $b{S}_4^2\geq 1$,
		\[f(x)=\frac{a}{2}+\frac{b}{4}x^2-\frac{{S}_4^{-2}}{4}x^2=\frac{a}{2}+\frac{1}{4}\left(b-\frac{1}{{S}_4^2}\right)x^2\geq \frac{a}{2}.\]
		
		If $N>4$, it is immediately seen that $f$ attains its minimum at  $$x_0=\left(\frac{2^\star }{2(2^\star-2)}{S}_N^{\frac{2^\star}{2}}b\right)^{\frac{1}{2^\star-4}}$$ and the claim is a consequence of the assumption $\displaystyle a^\frac{N-4}{2}b\geq C_1(N)$.
		
		Thus,  $$\liminf_{n\to \infty}(\mathcal{F}(u_n)-\mathcal{F}(u))\geq \liminf_{n \to \infty}\|u_n-u\|^2 f(\|u_n-u\|)\geq 0,$$ and the thesis follows.
	\end{proof}
	
	\begin{rem}We point out that the constant $C_1(N)$ in Lemma \ref{semicontinuity} is optimal, i.e. if
		$ a^{\frac{N-4}{2}} b< C_1(N)$  the functional $\mathcal F$ is no longer sequentially weakly lower semicontinuous (see \cite{FFK}).
		\end{rem}
 	
 In the next lemma we prove the Palais Smale property for our energy functional. Notice that the same constraints on $a$ and $b$ appear in \cite{H1} where such  property was investigated  for the critical Kirchhoff equation on closed manifolds by employing the $H^1$ (which is the underlying Sobolev space) decomposition.
 	\begin{lem}\label{Palais Smale}
 		Let $N \ge 4$ and $a,b$ be positive numbers such that $a^{\frac{N-4}{2}}b>C_{2}(N)$. For $\lambda>0, v^*\in H_{0}^{1}(\Omega)$ denote by $\mathcal{E}:H_{0}^{1}(\Omega)\to\mathbb{R}$ the  functional
 		defined by
 		\[
 		\mathcal{E}(u)=\frac{a}{2}\|u\|^{2}+\frac{b}{4}\|u\|^{4}-\frac{1}{2^{\star}}\|u\|_{2^{\star}}^{2^{\star}}-\frac{\lambda}{2}\|u-v^{\star}\|_{2}^{2}
 		 \qquad \mbox{for every }\  u \in H^1_0(\Omega).\] Then, $\mathcal E$
 		satisfies the Palais-Smale (shortly (PS)) condition.
 	\end{lem}
 	
 	\begin{proof}
 	Let $\{u_{n}\}$ be a  (PS) sequence for $\mathcal E$, that is
 		\[
 		\begin{cases}
 		\mathcal{E}(u_{n})\to c\\
 		\mathcal{E}'(u_{n})\to0
 		\end{cases}\mbox{as }n\to\infty.
 		\]
 	Since $\mathcal E$ is coercive,  $\{u_{n}\}$ is bounded and there exists $u\in H_{0}^{1}(\Omega)$ such that (up to a subsequence)
 		
 		\begin{align*}
 		u_{n} & \rightharpoonup u\mbox{ in }H_{0}^{1}(\Omega),\\
 		u_{n} & \to u\mbox{ in }L^{p}(\Omega),\ p\in[1,2^{\star}),\\
 		u_{n} & \to u\mbox{ a.e. in }\Omega.
 		\end{align*}
 		Using the second concentration compactness lemma of Lions \cite{L}, there exist an at most countable index set $J$,
 		a set of points $\{x_{j}\}_{j\in J}\subset\overline\Omega$ and two families of positive
 		numbers $\{\eta_{j}\}_{j\in J}$, $\{\nu_{j}\}_{j\in J}$ such that
 		\begin{align*}
 		|\nabla u_{n}|^{2} & \rightharpoonup d\eta\geq|\nabla u|^{2}+\sum_{j\in J}\eta_{j}\delta_{x_{j}},\\
 		|u_{n}|^{2^\star} & \rightharpoonup d\nu=|u|^{2^\star}+\sum_{j\in J}\nu_{j}\delta_{x_{j}},
 		\end{align*}
 		 (weak star convergence in the sense of measures), where $\delta_{x_{j}}$ is the Dirac mass concentrated at
 		$x_{j}$ and such that
 $$		S_{N}  \nu_{j}^{\frac{2}{2^\star}}\leq\eta_{j} \qquad \mbox{for every $j\in J$}.$$
 Next, we will prove that the index set $J$ is empty. Arguing
 		by contradiction, we may assume that there exists a $j_{0}$ such
 		that $\nu_{j_{0}}\neq0$. Consider now, for $\varepsilon>0$ a non negative  cut-off function $\phi_\varepsilon$ such that
 		\begin{align*}
 		&\phi_{\varepsilon}  =1\mbox{ on }B(x_{0},\varepsilon),\\
 		&\phi_{\varepsilon}  =0\mbox{ on } \Omega\setminus B(x_{0},2\varepsilon),\\
 		&|\nabla\phi_{\varepsilon}|  \leq\frac{2}{\varepsilon}.
 		\end{align*}
 		It is clear that the sequence $\{u_{n}\phi_{\varepsilon}\}_{n}$ is
 		bounded in $H_{0}^{1}(\Omega)$,  so that
 		\[
 		\lim_{n\to\infty}\mathcal{E}'(u_{n})(u_{n}\phi_{\varepsilon})=0.
 		\]
 		Thus
 		\begin{align}\label{calc 1}
 		o(1) & =(a+b\|u_{n}\|^{2})\int_{\Omega}\nabla u_{n}\nabla(u_{n}\phi_{\varepsilon})-\int_{\Omega}|u_{n}|^{2^\star}\phi_{\varepsilon}-\lambda\int_{\Omega}(u_{n}-v^{*})(u_{n}\phi_{\varepsilon}) \nonumber \\
 		& =(a+b\|u_{n}\|^{2})\left(\int_{\Omega}|\nabla u_{n}|^{2}\phi_{\varepsilon}+\int_{\Omega}u_{n}\nabla u_{n}\nabla\phi_{\varepsilon}\right)-\int_{\Omega}|u_{n}|^{2^\star}\phi_{\varepsilon}-\lambda\int_{\Omega}(u_{n}-v^{*})(u_{n}\phi_{\varepsilon}).
 		\end{align}
 		Moreover, using  H\"{o}lder inequality, one has
 		\[
 		\left|\int_{\Omega}(u_{n}-v^{*})(u_{n}\phi_{\varepsilon})\right|\leq \left(\int_{B(x_{0},2\varepsilon)}(u_{n}-v^{*})^2\right)^\frac{1}{2} \left(\int_{B(x_{0},2\varepsilon)}u_n^2\right)^\frac{1}{2},
 		\] so that
 \[\lim_{\varepsilon\to0}\lim_{n\to\infty}\int_{\Omega}(u_{n}-v^{*})(u_{n}\phi_{\varepsilon})=0.\]
 Also,
 \begin{eqnarray*}
 \left|\int_\Omega u_{n}\nabla u_{n}\nabla\phi_{\varepsilon}\right|&=&\left|\int_{B(x_{0},2\varepsilon)}u_{n}\nabla u_{n}\nabla\phi_{\varepsilon}\right|\leq \left(\int_{B(x_{0},2\varepsilon)}|\nabla u_n|^2\right)^\frac{1}{2}
 \left(\int_{B(x_{0},2\varepsilon)}|u_n\nabla \phi_\varepsilon|^2\right)^\frac{1}{2}\\
  &\leq&  C \left(\int_{B(x_{0},2\varepsilon)}|u_n\nabla \phi_\varepsilon|^2\right)^\frac{1}{2}.
  \end{eqnarray*}
 Since $$\lim_{n\to\infty}\int_{B(x_{0},2\varepsilon)}|u_n\nabla \phi_\varepsilon|^2=\int_{B(x_{0},2\varepsilon)}|u\nabla \phi_\varepsilon|^2,$$ and
  \begin{eqnarray*}
   \left(\int_{B(x_{0},2\varepsilon)}|u\nabla \phi_\varepsilon|^2\right)^\frac{1}{2}&\leq &
 \left(\int_{B(x_{0},2\varepsilon)} |u|^{2^\star}\right)^\frac{1}{2^\star}
 \left(\int_{B(x_{0},2\varepsilon)}|\nabla \phi_\varepsilon|^N \right)^\frac{1}{N}\\
 &\leq& C \left(\int_{B(x_{0},2\varepsilon)} |u|^{2^\star}\right)^\frac{1}{2^\star}
\end{eqnarray*}
 we get

 		\[
 		\lim_{\varepsilon\to0}\lim_{n\to\infty}(a+b\|u_{n}\|^{2})\left|\int_\Omega u_{n}\nabla u_{n}\nabla\phi_{\varepsilon}\right|=0.
 		\]
 		Moreover, as $0\leq \phi_\varepsilon\leq 1$,
 \begin{eqnarray*}
 \lim_{n\to\infty}(a+b\|u_{n}\|^{2})\int_{\Omega}|\nabla u_{n}|^{2}\phi_{\varepsilon}&\geq&
 \lim_{n\to\infty}\left[a\int_{B(x_{0},2\varepsilon)}|\nabla u_{n}|^{2}\phi_{\varepsilon}+b\left(\int_{\Omega}|\nabla u_{n}|^{2}\phi_{\varepsilon}\right)^{2}\right]\\&\geq&
a\int_{B(x_{0},2\varepsilon)}|\nabla u|^{2}\phi_{\varepsilon}+b\left(\int_{\Omega}|\nabla u|^{2}\phi_{\varepsilon}\right)^{2}+a\eta_{j_{0}}+b\eta_{j_{0}}^{2}.
 \end{eqnarray*}
 So, as $\int_{B(x_{0},2\varepsilon)}|\nabla u|^{2}\phi_{\varepsilon}\to 0$ as $\varepsilon\to 0$,
 		\[
 		\lim_{\varepsilon\to0}\lim_{n\to\infty}(a+b\|u_{n}\|^{2})\int_{\Omega}|\nabla u_{n}|^{2}\phi_{\varepsilon} \geq a\eta_{j_{0}}+b\eta_{j_{0}}^{2}.\]
 		
 		Finally,
 		\begin{align*}
 		\lim_{\varepsilon\to0}\lim_{n\to\infty}\int_\Omega|u_{n}|^{2^\star}\phi_{\varepsilon} & =\lim_{\varepsilon\to0}\int_\Omega |u|^{2^\star}\phi_{\varepsilon}+\nu_{j_{0}}=\lim_{\varepsilon\to0}\int_{B(x_{0},2\varepsilon)} |u|^{2^\star}\phi_{\varepsilon}+\nu_{j_{0}}=\nu_{j_{0}}.
 		\end{align*}
 		Summing up the above outcomes, from \eqref{calc 1}
 		one obtains
 		\begin{align*}
 		0 & \geq a\eta_{j_{0}}+b\eta_{j_{0}}^{2}-\nu_{j_0}\geq a\eta_{j_{0}}+b\eta_{j_{0}}^{2}-S_{N}^{-\frac{2^\star}{2}}\eta_{j_{0}}^{\frac{2^\star}{2}}\\
 		& =\eta_{j_{0}}\left(a+b\eta_{j_{0}}-S_{N}^{-\frac{2^\star}{2}}\eta_{j_{0}}^{\frac{2^\star-2}{2}}\right).
 		\end{align*}
 		Denote by $f_{1}:[0,+\infty[\to\mathbb{R}$ the function ${\displaystyle f_{1}(x)=a+bx-S_{N}^{-\frac{2^\star}{2}}x^{\frac{2^\star-2}{2}}}$.
 		As before, assumptions on $a$ and $b$ imply that $f_{1}(x)>0$ for all $x\geq0$. Thus
 		\[
 		a+b\eta_{j_{0}}-S_{N}^{-\frac{2^\star}{2}}\eta_{j_{0}}^{\frac{2^\star-2}{2}}>0,
 		\]
 		therefore $\eta_{j_{0}}=0,$ which is a contradiction. Such conclusion  implies
 		that $J$ is empty, that is
 		\[\lim_{n\to\infty}\int_{\Omega}|u_n|^{2^\star}= \int_{\Omega}|u|^{2^\star}\]
and the uniform convexity of $L^{2^\star}(\Omega)$ implies that 		\[
 		u_{n}\to u\mbox{ in }L^{2^\star}(\Omega).
 		\]
 		Now, recalling that the derivative of the function $$u\to \frac{a}{2}\|u\|^{2}+\frac{b}{4}\|u\|^{4}$$  satisfies the $(S_+)$ property, in  a standard way one can see that $u_{n}\to u\mbox{ in }H_{0}^{1}(\Omega)$, which proves
 		our lemma.
 	\end{proof}
 	
 In the proof of our  result, the main tool is the following theorem:

 \begin{theor}[Ricceri \cite{R1}, Theorem 3.2]\label{minimax theorem}
 	Let $X$ be a topological space, $E$ a real Hausdorff
 	topological vector space, $C\subseteq E$ a convex set,
 	$f : X\times C \to \R$ a function which is lower semicontinuous,
 	inf--compact in $X$, and upper semicontinuous and concave in $C$. Assume also that
\begin{equation}\label{minimax}
\sup_{v\in C}\inf_{x\in X}f(x,v)<\inf_{x\in X}\sup_{v\in C} f(x,v).
\end{equation}
Then, there exists $v^*\in C$ such that the function $f(\cdot, v^*)$ has at least two global
minima.
 	\end{theor}
\noindent {\bf Proof of Theorem \ref{our theorem}}
We apply Theorem \ref{minimax theorem} with $X=H^1_0(\Omega)$ endowed with the weak topology, $E=L^2(\Omega)$ with the strong topology, $C$ as in the assumptions.
Let $\mathcal F$ as in Lemma \ref{semicontinuity}, i.e.
\[\mathcal F(u)=\frac{a}{2}\|u\|^2+\frac{b}{4} \|u\|^4-\frac{1}{2^\star}\|u\|^{2^\star}_{2^\star} \qquad \mbox{for every }\  u \in H^1_0(\Omega).\]
From Lemma \ref{semicontinuity}, $\mathcal F$ is sequentially weakly lower semicontinuous, and coercive,  thus, the set $M_\mathcal F$ of its global minima is non empty.

Denote by
\begin{equation}\label{lambdastar}\lambda^\star=\inf\left\{\frac{\mathcal F(u)-\mathcal F(v)}{\|u-v\|_2^2}  \ : \ (v, u)\in M_{\mathcal F}\times H^1_0(\Omega), \  v\neq u \right\}
\end{equation} and fix $\lambda>\lambda^\star$.

Let $f: H^1_0(\Omega)\times C\to\R $ be the function
\[f(u,v)=\mathcal F(u)-\lambda \|u-v\|_2^2.\]

From the Eberlein Smulyan theorem it follows that $f(\cdot, v)$ has weakly compact sublevel sets in $H^1_0(\Omega)$. It is also clear that $f(u, \cdot)$ is continuous and concave in $L^2(\Omega)$. Let us prove \eqref{minimax}.

Recalling that the closure of $C$ in $L^2(\Omega)$ (denoted by ${\overline C}$) contains $H^1_0(\Omega)$, one has
\begin{align}\label{first}
\inf_{u\in H^1_0(\Omega)}\sup_{v\in C } f(u,v)&=\inf_{u\in H^1_0(\Omega)}\sup_{v\in \overline {C}}f(u,v)\nonumber \\&\geq
\inf_{u\in H^1_0(\Omega)}\sup_{v\in H^1_0(\Omega) } f(u,v)\nonumber\\&=
\inf_{u\in H^1_0(\Omega)}\sup_{v\in H^1_0(\Omega) } (\mathcal F(u)-\lambda \|u-v\|_2^2)\nonumber \\\nonumber&=\inf_{u\in H^1_0(\Omega)}(\mathcal F(u)-\lambda \inf_{v\in H^1_0(\Omega)}\|u-v\|_2^2)\\&=
\min_ {H^1_0(\Omega)}\mathcal F
\end{align}

Since  $\lambda>\lambda^\star$, there exist $u_0, v_0\in  H^1_0(\Omega), u_0\neq v_0$ and $\varepsilon>0 $ such that
\begin{align*}
&\mathcal F(u_0)-\lambda \|u_0-v_0\|_2^2<\mathcal F(v_0)-\varepsilon,\\
& \mathcal F(v_0)= \min_ {H^1_0(\Omega)}\mathcal F.
\end{align*}
Thus, if $h:L^2(\Omega)\to\R$ is the function defined by $h(v)=\inf_{u\in H^1_0(\Omega)}(\mathcal F(u)-\lambda \|u-v\|_2^2)$,
then, $h$ is upper semicontinuous in $L^2(\Omega)$ and \[h(v_0)\leq \mathcal F(u_0)-\lambda \|u_0-v_0\|_2^2<\mathcal F(v_0)-\varepsilon.\]
So, there exists $\delta>0$ such that $h(v)<\mathcal F(v_0)-\varepsilon$ for all $\|v-v_0\|_2\leq \delta.$
Therefore,
\[\sup_{\|v-v_0\|_2\leq \delta }\inf_{u\in H^1_0(\Omega)}(\mathcal F(u)-\lambda \|u-v\|_2^2)\leq \mathcal F(v_0)-\varepsilon.\]
On the other hand,
 \[\sup_{\|v-v_0\|_2\geq \delta }\inf_{u\in H^1_0(\Omega)}(\mathcal F(u)-\lambda \|u-v\|_2^2)\leq
\sup_{\|v-v_0\|_2\geq \delta } (\mathcal F(v_0)-\lambda \|v_0-v\|_2^2)\leq \mathcal F(v_0)-\lambda\delta^2.\]
Summing up the above outcomes, we obtain

\begin{align}\label{second}
\sup_{v\in C }\inf_{u\in H^1_0(\Omega)} f(u,v)&\leq \sup_{v\in L^2(\Omega) }\inf_{u\in H^1_0(\Omega)}f(u,v)\nonumber \\&=
\sup_{v\in L^2(\Omega) }\inf_{u\in H^1_0(\Omega)}(\mathcal F(u)-\lambda \|u-v\|_2^2)\nonumber\\&<\mathcal F(v_0)=\min_{H^1_0(\Omega)}\mathcal F.
\end{align}
From \eqref{first} and \eqref{second},  claim \eqref{minimax} follows.
Applying Theorem \ref{minimax theorem}, we deduce the existence of $v^*\in C$ such that  the energy functional
\[\mathcal E(u)=\mathcal F(u)-\frac{\lambda}{2}\|u-v^*\|_2^2\] associated to our problem has two global minima, which is claim $(A)$. In order to prove $(B)$ we observe that, since  the functional is of class $C^1$, such global minima turns out to be weak solutions of our problem. The third solution follows by Lemma \ref{Palais Smale} (recall that $C_2(N)\geq C_1(N)$) and  a classical version of the Mountain Pass theorem by Pucci and Serrin \cite{PS}.\qed

\begin{rem} For sake of clarity, we calculate the approximate
values of the constants $C_1(N)$ and $C_2(N)$ for some  $N$:
	\begin{center}
	\begin{tabular}{|c|c|c|}
	\hline
	$N$ & $C_1(N)$ & $C_{2}(N)$\tabularnewline
	\hline
	\hline
	5 & 0.002495906672 & 0.002685168050\tabularnewline
	\hline
	6 & 0.0001990835458 & 0.0002239689890\tabularnewline
	\hline
	7 & 0.00001712333233 & 0.00001985538802\tabularnewline
	\hline
	9 & 1.269275934$\cdot10^{-7}$ & 1.529437355$\cdot10^{-7}$\tabularnewline
	\hline
\end{tabular}
	
\end{center}

\end{rem}

	{\begin{question}
	 Notice that if $N=4$ then, for $b S_N^2< 1$, $\mathcal E$ is unbounded from below. Indeed, if $\{u_n\}$ is such that $\frac{\|u_n\|^2}{\|u_n\|_4^2}\to S_N$, then we can fix $c$ and $\bar n$ such that $\frac{\|u_{\bar n}\|^2}{\|u_{\bar n}\|_4^2}<c<b^{-\frac{1}{2}}$. Thus
	 \[
	 \mathcal{E}(\tau u_{\bar n})<\frac{a\tau^2}{2}\|u_{\bar n}\|^{2}+\frac{\tau^4}{4}\left(b-\frac{1}{c^2}\right)\|u_{\bar n}\|^{4}-\frac{\lambda}{2}\|\tau u_{\bar n}-v^{*}\|_{2}^{2}\to-\infty, \mbox{as} \ \tau\to+\infty.
	 \] It remains an open question if, when $N>4$, Theorem \ref{our theorem} holds for every $a\geq 0, b>0$ with $ a^{\frac{N-4}{2}} b< C_1(N)$.
	\end{question}
{\bf Acknowledgment} This work was initiated when Cs.
Farkas visited the Department of Mathematics of the University of
Catania, Italy. He thanks the financial support of Gruppo Nazionale per l'Analisi Matematica, la Probabilit\`a e le loro Applicazioni (GNAMPA) of the Istituto Nazionale di Alta Matematica (INdAM).

\end{document}